\documentclass{birkjour}

\usepackage[utf8]{inputenc}
\usepackage[T1]{fontenc}
\usepackage{lmodern}
\usepackage{amssymb}
\usepackage{array}
\usepackage{mathtools}
\usepackage{microtype}
\usepackage[hidelinks]{hyperref}
\usepackage{tabularx}
\usepackage{url}
\usepackage{xurl}
\usepackage{xspace}
\renewcommand{\subjclassname}{Mathematics Subject Classification (2020)}

\theoremstyle{plain}
\newtheorem{theorem}{Theorem}[section]
\newtheorem{proposition}[theorem]{Proposition}
\newtheorem{lemma}[theorem]{Lemma}
\newtheorem{corollary}[theorem]{Corollary}

\theoremstyle{definition}
\newtheorem{definition}[theorem]{Definition}

\theoremstyle{remark}

\numberwithin{equation}{section}

\newcommand{\Cl}{\mathrm{Cl}}
\newcommand{\Spin}{\mathrm{Spin}}

\newcommand{\SO}{\mathrm{SO}}
\newcommand{\End}{\mathrm{End}}

\newcommand{\Hyp}{\mathrm{H}}

\newcommand{\wedgeEven}{\bigwedge^{\mathrm{even}} W}
\newcommand{\wedgeOdd}{\bigwedge^{\mathrm{odd}} W}
\newcommand{\SpinorLean}{\textsc{SpinorLean}\xspace}
\newcommand{\Lean}{Lean~4\xspace}
\newcommand{\Mathlib}{\textsc{Mathlib}\xspace}
\DeclareRobustCommand{\code}[1]{\ifmmode\text{\nolinkurl{#1}}\else\nolinkurl{#1}\fi}
\newcolumntype{Y}{>{\raggedright\arraybackslash}X}

\title[Exterior-Model Spinors in Split Rank]
      {Exterior-Model Spinors in Split Rank:\\
       Exact Levi Images and Square-Determinant Obstructions}

\author{Arthur F.~Ramos}
\address{Microsoft, USA}
\email{arfreita@microsoft.com}
\thanks{Corresponding author.}

\author{David B.~Hulak}
\address{Independent Researcher}
\email{dbhulak@gmail.com}

\author{Ruy J.~G.~B.~de Queiroz}
\address{Centro de Inform\'atica,
         Universidade Federal de Pernambuco, Brazil}
\email{ruy@cin.ufpe.br}

\date{}

\begin{document}

\begin{abstract}
Let \(K\) be a field with \(2\in K^\times\), and let
\(\Hyp(W)\) denote the standard hyperbolic form on \(W^*\oplus W\). We study
the exterior spinor model \(S=\bigwedge W\) together with the spin-to-orthogonal
map for this split form, keeping the chosen hyperbolic presentation explicit.

The main results determine the field-sensitive part of the split Levi image. In
positive split rank the kernel of \(\Spin(V,Q)\to\SO(V,Q)\) is \(\{\pm1\}\);
therefore the exterior spinor action descends to the orthogonal image only
projectively. For the split line the image of
\(\Spin(\Hyp(K))\to\SO(\Hyp(K))\) is precisely the square-scaling
subgroup. In arbitrary split rank we construct explicit Clifford representatives
for hyperbolic transvections and chosen-line square scalings, prove the
weight-\(2\) torus conjugation law, and show that any split Levi lift acts on
\(\bigwedge W\) as a scalar multiple of the natural exterior action. If
\(\det(g)=u^2\), the transported Levi element
\(\Lambda(g)=(g^{-{\vee}},g)\) admits an explicit even unitary Clifford lift
acting as \(-u^{-1}\bigwedge g\) on \(S\). In finite split rank at least three,
if
\[
  H_W=\operatorname{im}\bigl(\Spin(\Hyp(W))\to\SO(\Hyp(W))\bigr),
\]
then
\[
  \Lambda(g)\in H_W
  \quad\Longleftrightarrow\quad
  \det(g)\in (K^\times)^2.
\]
Equivalently, the spin image meets the split Levi subgroup exactly in its
square-determinant subgroup. This recovers, by direct Clifford calculation, the
determinant-modulo-squares spinor-norm criterion on the split Levi.
\end{abstract}

\maketitle

\section{Introduction}
\label{sec:intro}

The exterior construction of spinors is classical. Once a maximal totally
isotropic subspace \(W\) of a hyperbolic quadratic space has been chosen, the
spinor module is \(S=\bigwedge W\), Clifford generators act by exterior
multiplication and contraction, and the even Clifford algebra preserves the
parity decomposition. Over a general field, however, the image of the spin group
inside the orthogonal group depends on square classes in \(K^\times\). The
chosen exterior model provides a convenient setting in which those square-class
obstructions can be computed explicitly.

This paper keeps the hyperbolic presentation and the associated Levi embedding
fixed throughout the calculation. The resulting statements separate the
classical exterior model from the field-sensitive image problem: the former
provides the representation, while the latter is governed by explicit Clifford
representatives and by determinant classes in \(K^\times/(K^\times)^2\).

Starting from an explicit hyperbolic presentation
$Q \simeq \Hyp(W)$, we construct the spinor
module $S=\bigwedge W$, transport it through Witt-style reformulations, and
obtain the induced $\Spin(V,Q)$-action. A crucial point is that our
spinor module is \emph{not} the ambient exterior algebra on $V$, but
the exterior model $\bigwedge W$ determined by a chosen hyperbolic
presentation---the object with the classical spin dimension and
chiral decomposition.

The image statements require all maps to be expressed in the same hyperbolic
coordinates. We therefore keep the chosen model, root elements, torus elements,
and Levi embedding in a single notation. This gives direct proofs of the
Levi-action bridge, the exact split-line image theorem, the higher-rank
square-determinant factorization with its Clifford lift, and the exact
square-determinant criterion for the split Levi spin image in finite split rank
at least three.

\paragraph{Contributions.}
The main contributions are as follows. First, we give explicit Clifford lifts of
hyperbolic transvections and chosen-line square scalings and compute their
orthogonal and exterior actions. Second, we prove that any split Levi lift acts
projectively as the exterior action \(\bigwedge g\). Third, we prove exact image
statements: the split-line image is the square-scaling subgroup, and in split
rank at least three the image on the split Levi is precisely the
square-determinant subgroup. The classical chosen-model equivalence
\(\Cl(\Hyp(W))\simeq \End(\bigwedge W)\) and the half-spin decomposition
are included as the algebraic infrastructure used by these image calculations.

\paragraph{Standing conventions.}
Unless explicitly specialized to $\mathbb{R}$ or $\mathbb{C}$, the paper works
over a field $K$ with $2\in K^\times$. The notation $\Hyp(W)$ always
denotes the standard hyperbolic form on $W^* \times W$. Chosen-model theorems
are stated either for an explicit hyperbolic presentation
$e:Q \simeq \Hyp(W)$ or, in the split-rank reformulation obtained from a
chosen Witt decomposition, for a
finite-dimensional nondegenerate form $Q$ with
$\dim V = 2\,\mathrm{wittIndex}(Q)$. Statements about the negative half-spin
module beyond its formal zero-rank occurrence, and statements about equal
half-spin dimensions, additionally assume positive split rank.
We write $\tilde x$ for Clifford conjugation (grade involution followed by
reversion), so $\tilde v=-v$ for vectors and $\widetilde{vw}=wv$ for vectors
\(v,w\). For the mathematical statements we use the norm-one even Lipschitz
convention
\[
  \Spin(V,Q)=
  \{\,x\in\Cl^0(V,Q)^\times:
       xVx^{-1}=V,\ x\tilde x=\tilde x x=1\,\}.
\]
For finite-dimensional nondegenerate forms this is the standard spin group and
agrees with the product-of-anisotropic-vectors model
\cite{Artin1957GeometricAlgebra,OMeara2000QuadraticForms}. The explicit
spin generators used below are products of two vectors of \(Q\)-value \(-1\);
they preserve \(V\) by the reflection formula and have norm one, hence belong to
\(\Spin(V,Q)\). The Lean companion uses Mathlib's \(\Spin\) group;
the image-inclusion proofs are witnessed by the same explicit
pair-reflection products, so no appeal to a converse---from vector preservation
to products of vectors---is needed.
All statements about the map $\Spin(V,Q)\to\SO(V,Q)$ concern the corresponding
groups of $K$-points; no stronger scheme-theoretic claim is intended.

The split line is the smallest test case. When $\dim W=1$, the exact image
theorem shows that $\Spin(\Hyp(W)) \to \SO(\Hyp(W))$ hits precisely the
square-scaling subgroup; hence the split-line double cover is surjective if and
only if every unit of \(K\) is a square. The higher-rank results below extend
this square-class obstruction from the line to the split Levi subgroup: square
determinant gives a chosen-model Clifford lift with normalized exterior action,
and in finite split rank at least three the converse also holds.

\section{Context}
\label{sec:related}

\paragraph{Mathematical background.} Our treatment follows the classical
presentations of Lawson--Michelsohn~\cite{LawsonMichelsohn1989},
Chevalley~\cite{Chevalley1997}, Atiyah--Bott--Shapiro~\cite{AtiyahBottShapiro1964},
and Lounesto~\cite{Lounesto2001}, with the low-dimensional octonionic and
quaternionic incidences drawn from Baez~\cite{Baez2002}. The underlying
``choose a maximal isotropic and use $\bigwedge W$'' construction is classical;
the additional content here is the explicit computation of the spin image on
split Levi factors over arbitrary fields.

\paragraph{Companion verification.}
The calculations have been checked in a Lean 4/\Mathlib companion development
\cite{deMouraUllrich2021}, using \Mathlib's Clifford and exterior algebra libraries
\cite{Mathlib,MathlibCommunity2020}. The source for the companion development is
the \SpinorLean repository~\cite{SpinorLeanRepo}. The formal artifact is used as
independent verification of the coordinate identities and subgroup inclusions;
the mathematical statements and proofs below are self-contained. Relevant formal
antecedents include Wieser and Song's Lean geometric-algebra formalization
\cite{Wieser2022Clifford} and the earlier \Lean~3 \texttt{lean-ga}
library~\cite{LeanGA}.

\paragraph{Scope of the contribution.} We do not claim novelty for the exterior-model
realization $\Cl(Q)\simeq\End(\bigwedge W)$, the parity split, or the basic
chosen-model reconstruction of spinors. The nonclassical content begins when the
covering and image behavior is made field-sensitive in that chosen model: the
positive split-rank linear/projective dichotomy; the explicit transvection and
chosen-line square-scaling lifts; their internal and orthogonal weight-$2$
semidirect structure together with the exact normalized exterior action of the
torus lift; the theorem that split Levi lifts act projectively as the natural
exterior action; the higher-rank square-determinant factorization and its
explicit even unitary chosen-model lift; the finite-rank-at-least-three exact
split-Levi spin-image criterion; and the exact split-line image theorem. The
result is a chosen-model realization of root subgroups, the square torus, and
the square-class
obstruction. Although the determinant modulo squares criterion agrees with the
classical spinor norm, the point here is to recover it directly inside the
chosen exterior model, with explicit Clifford representatives and with the
induced exterior action computed throughout.

\paragraph{Main theorems.}
The principal statements are the structural Levi-action bridge, the exact
split-line image theorem, and the exact higher-rank square-determinant image
criterion on the split Levi. More explicitly, the paper proves:
\begin{enumerate}
\item the structural Levi-action theorem
  (Theorem~\ref{thm:levi-projective-exterior}), which identifies any split Levi
  lift projectively with the natural exterior action;
\item the exact split-line image theorem of
  Proposition~\ref{prop:split-line-square};
\item the higher-rank square-determinant Levi factorization and explicit
  chosen-model lift of
  Theorem~\ref{thm:square-det-levi};
\item the finite-rank-at-least-three exact image theorem
  (Theorem~\ref{thm:square-det-levi-image}),
  which proves that a split Levi element lies in the spin image if and only if
  its determinant is a square unit;
\item the explicit hyperbolic transvection lift
  $x_{\delta,w}=1+\iota(\delta,0)\iota(0,w)$ from
  Proposition~\ref{prop:transvection-unit} together with the explicit chosen-line
  square-scaling lift of Proposition~\ref{prop:chosen-line-square}, its internal
  Clifford-level torus conjugation law, and its exact chosen-model action
  $\rho_S(u_t v)=-t^{-1}\bigwedge \ell_t$
  (Corollary~\ref{cor:internal-clifford-torus}), and the corresponding
  orthogonal semidirect package
  (Corollaries~\ref{cor:torus-transvection} and
  \ref{cor:one-line-semidirect});
\item the positive split-rank kernel/nonfactorization/projective-descent package
  from Theorem~\ref{thm:kernel-nonfactor-split-line} and
  Proposition~\ref{prop:projective-descent}.
\end{enumerate}
The exterior-model equivalences and half-spin decomposition theorems provide
the classical foundation on which these field-sensitive
statements rest.

\paragraph{Logical structure.}
The exact higher-rank theorem is obtained from five calculations: the
matrix-model kernel calculation \(\ker(\Spin\to\SO)=\{\pm1\}\), the direct
split-line image computation, the vacuum-line proof that Levi lifts act
projectively as exterior maps, the square-determinant factorization with its
torus lift, and the one-line projector argument that extracts the determinant
square class from spin unitarity.

\section{Chosen-model interface}
\label{sec:arch}

This section records the classical chosen-model framework used in the
field-sensitive image calculations.

\begin{definition}[Hyperbolic presentation and chosen spinor model]
A hyperbolic presentation of a finite-dimensional quadratic form $Q$ over $K$
is an isometry $Q \simeq \Hyp(W)$ with $W$ finite-dimensional, where
$\Hyp(W)$ denotes the standard hyperbolic form on $W^* \times W$. The
associated chosen spinor model is $S(W)=\bigwedge W$, with parity pieces
$S^\pm(W)=\bigwedge^{\mathrm{even}/\mathrm{odd}} W$.
\end{definition}

\paragraph{Hyperbolic presentation and transport.} The central organizing idea is
to isolate the data
\[
  Q \simeq \Hyp(W)
  \qquad (W \text{ a maximal totally isotropic subspace})
\]
as a specified choice of maximal isotropic subspace together with an isometry to
the standard hyperbolic form. From that single piece of data one obtains the
chosen exterior model $\bigwedge W$, its even and odd halves, the transported
Clifford action $\Cl(Q)\to\End(\bigwedge W)$, and the restricted spin-group
action. The decisive point is that our ``spinor module'' is not the
ambient regular module on $V$, but the distinguished exterior model $\bigwedge W$
determined by the presentation.

The construction then proceeds in three steps. First, the split model on
$W^* \times W$ is handled explicitly, with wedge and contraction giving the
Clifford action and explicit basis projectors giving
$\Cl(\Hyp(W))\simeq\End(\bigwedge W)$. Second, any hyperbolic
presentation transports that split result to the target quadratic space and
its spin action. Third, Witt decomposition and the split-rank reformulation recover
top-level statements that no longer require users to carry an explicit isometry
argument. In this way, each substantive theorem is proved once on the split
model and then reused at the levels where classification and orthogonal-action
results are stated.

\section{Main Results}
\label{sec:results}

Throughout this section, $K$ is a field with $2\in K^\times$ and $Q$ is a
finite-dimensional quadratic form. The algebra equivalences are first stated
for an explicit hyperbolic presentation $e:Q \simeq \Hyp(W)$ with
$W$ finite-dimensional; their split-rank reformulations additionally assume
that $Q$ is nondegenerate and $\dim V = 2\,\mathrm{wittIndex}(Q)$. Statements
involving a nontrivial negative half or the equal-dimension formulas
additionally assume $\dim W>0$ (equivalently positive split rank in the
split-rank reformulation). The kernel theorem uses the finite-dimensional
nondegenerate split-rank setting.

\subsection*{Classical chosen-model foundation}

The first theorem collects the standard exterior-model material that the later
covering and image statements require. It is included as foundational setup, not
as the primary contribution of this paper.

\begin{theorem}[Exterior-model spinor theorem]
\label{thm:chosen-spinor-package}
Assume $K$ is a field with $2\in K^\times$, let
$e:Q \simeq \Hyp(W)$ be an explicit hyperbolic presentation, and set
$S=\bigwedge W$, $S^+=\wedgeEven$, and $S^-=\wedgeOdd$. Then the following hold.
\begin{enumerate}
\item For the explicit hyperbolic presentation, the Clifford action on $S$ by
  wedge and contraction yields
  \[
    \Cl(Q)\simeq \End(S).
  \]
  If $\dim W>0$, then the even action yields
  \[
    \Cl^+(Q)\simeq \End(S^+) \times \End(S^-);
  \]
  if $\dim W=0$, then $\Cl^+(Q)=K=\End(S^+)$ and $S^-=0$.
\item If $\dim W=n$, then $\dim S = 2^n$; if $n>0$, then
  $\dim S^+ = \dim S^- = 2^{n-1}$.
\item For finite-dimensional nondegenerate split-rank $Q$ with
  $\dim V = 2\,\mathrm{wittIndex}(Q)$, any chosen Witt decomposition produces a
  hyperbolic presentation and hence a spinor model with the preceding
  properties. For each such choice, the positive half-spin module is simple. If
  the split rank is positive, then the negative half-spin module is simple and
  the two half-spin modules are inequivalent as modules over the even Clifford
  algebra; if the split rank is $0$, then $S^+=K$ and $S^-=0$.
\end{enumerate}
\end{theorem}

\begin{proof}
The explicit-presentation statement is obtained in
Section~\ref{sec:highlights} by transporting the split-model matrix-unit
calculation (Theorem~\ref{thm:matrix-units-split}) and its parity corollary
(Corollary~\ref{cor:even-half-spin}) along
Proposition~\ref{prop:transport-split}. The dimension formulas are the standard
binomial counts for $\bigwedge W$, and the split-rank formulation follows by
choosing a Witt decomposition and applying the same transported package. We
state the result here because the later covering and image theorems are all
built on this chosen-model interface.
\end{proof}

We deliberately separate the presentation-dependent and split-rank
versions of these results. The endomorphism-algebra equivalences are
first proved from an explicit hyperbolic presentation, where no separate
nondegeneracy hypothesis is needed; the nondegenerate split-rank assumptions
enter when the explicit isometry data are suppressed and the chosen model is
recast through a Witt construction.

At split rank $0$, one simply has $W=0$, $S^+=K$, and $S^-=0$. Positive-rank
hypotheses first enter when one asks for a nontrivial negative half-spin module,
for the equal half-spin dimension formula, or for the two-factor even-part
decomposition.

This theorem summarizes classical chosen-model material. The new field-sensitive
covering and image statements follow.

\subsection*{New field-sensitive covering and image package}

\begin{theorem}[Positive split-rank kernel, spinorial descent, and exact split-line image]
\label{thm:kernel-nonfactor-split-line}
Assume $K$ is a field with $2\in K^\times$. If $Q$ is finite-dimensional,
nondegenerate, and of positive split rank
\[
  0<\dim V = 2\,\mathrm{wittIndex}(Q),
\]
then the kernel of the spin-to-isometry map
\[
  \Spin(V,Q) \to \SO(V,Q)
\]
is exactly $\{\pm1\}$. If $S=\bigwedge W$ is the exterior-model spinor module
attached to any chosen Witt decomposition of $Q$, then the resulting spin
representation
\[
  \Spin(V,Q) \to \operatorname{GL}(S)
\]
does not factor through $\SO(V,Q)$: there is no group homomorphism
$\rho:\SO(V,Q)\to \operatorname{GL}(S)$ whose composite with
$\Spin(V,Q)\to\SO(V,Q)$ is the given spin action on $S$. The same is true for
each nonzero half-spin piece $S^\pm$. However, the induced action on the
projective spaces $\mathbb{P}(S)$ and $\mathbb{P}(S^\pm)$ depends only on the
image of $\Spin(V,Q)\to\SO(V,Q)$ and therefore defines an action of that image
subgroup. For the split line $Q=\Hyp(K)$, the image inside $\SO(Q)$ is
exactly the square-scaling subgroup, i.e.\ the subgroup of maps
\[
  (x,y)\longmapsto (a x, a^{-1} y)
  \qquad (a\in (K^\times)^2).
\]
Consequently, over fields with nonsquare units, the spin-to-orthogonal map on
the split line is not surjective; it becomes surjective exactly under
square-surjectivity of \(K^\times\).
\end{theorem}

\begin{proof}
Section~\ref{sec:highlights} proves the kernel calculation
(Proposition~\ref{prop:kernel-from-matrix}), non-factorization
(Proposition~\ref{prop:nonfactorization}), projective descent
(Proposition~\ref{prop:projective-descent}), and the exact split-line image
statement (Proposition~\ref{prop:split-line-square}). The final split-line
clause is exact because every element of $\SO(\Hyp(K))$ is a reciprocal
scaling and Proposition~\ref{prop:split-line-square} identifies precisely which
such scalings lie in the spin image.
\end{proof}

The next statements give the split Levi action and image theorems. Their proofs
are given in Section~\ref{sec:highlights}, after the root and torus
constructions used in the argument.

\begin{theorem}[Levi lifts act projectively as the exterior action]
\label{thm:levi-projective-exterior}
Let $Q=\Hyp(W)$ and let $S=\bigwedge W$ be the chosen exterior model.
Suppose $s\in \Spin(Q)$ projects to the transported Levi element
\[
  \Lambda(g)=(g^{-{\vee}},g)\in \SO(Q)
\]
for some $g\in \operatorname{GL}(W)$. Then there exists $c\in K^\times$ such that
\[
  \rho_S(s)=c\,(\bigwedge g)
\]
as endomorphisms of $S$. In particular, the projective action of a split Levi
lift on $S$ depends only on its Levi coordinate $g$.
\end{theorem}

\begin{theorem}[Square-determinant Levi factorization and chosen-model lift]
\label{thm:square-det-levi}
Let $S=\bigwedge W$. Assume \(W\) is finite-dimensional with $\dim W\ge 2$, and let
\[
  \Lambda:\operatorname{GL}(W)\hookrightarrow \SO(\Hyp(W))
\]
denote the transported Levi embedding $\Lambda(g)=(g^{-{\vee}},g)$. If
$g\in \operatorname{GL}(W)$ has $\det(g)=u^2$ for some $u\in K^\times$, then,
after choosing a basis \(e_1,\ldots,e_n\), there are finite products \(A\) and
\(B\) of elementary basis transvections and units \(t_2,\ldots,t_n\in K^\times\)
such that
\[
  g=A\,L_1(u^2)\,\prod_{j=2}^{n}D_{j1}(t_j)\,B.
\]
Here \(L_1(u^2)\) scales \(e_1\) by \(u^2\) and fixes the other basis vectors,
while \(D_{j1}(t_j)\) scales \(e_j\) by \(t_j\), scales \(e_1\) by \(t_j^{-1}\),
and fixes the remaining basis vectors. Applying \(\Lambda\), each elementary
transvection factor is a transported hyperbolic transvection and each
\(D_{j1}(t_j)\) is a product of four such transvections. Thus \(\Lambda(g)\) is a
product of transported hyperbolic transvections and one chosen-line square
scaling. Moreover, the factors admit explicit Clifford representatives whose
product is an even unitary Clifford unit
\[
  x\in \Cl(\Hyp(W))^\times,
\]
such that
\[
  \rho_S(x)=-u^{-1}(\bigwedge g)
\]
as endomorphisms of $S$. Consequently, $\Lambda(\operatorname{SL}(W))$ is
contained in the subgroup of
$\SO(\Hyp(W))$ generated by the transported hyperbolic transvections.
\end{theorem}

The preceding theorem constructs an
even unitary Clifford representative with the correct exterior action and
orthogonal factorization. The next theorem serves a different purpose: it gives the
exact membership criterion for the spin image. The forward direction avoids any converse from
vector-preserving Clifford units to products of vector generators; instead, it
writes the required root elements as explicit products of norm-\(-1\)
pair-reflection spin elements.

\begin{theorem}[Exact higher-rank square-determinant Levi image]
\label{thm:square-det-levi-image}
Assume \(W\) is finite-dimensional with $\dim W\ge 3$, and set
\[
  H_W=\operatorname{im}\bigl(\Spin(\Hyp(W))\to\SO(\Hyp(W))\bigr).
\]
For every \(g\in\operatorname{GL}(W)\),
\[
  \Lambda(g)\in H_W
  \quad\Longleftrightarrow\quad
  \det(g)\in (K^\times)^2.
\]
Equivalently, the intersection of the spin image with the split Levi subgroup is
exactly the square-determinant Levi subgroup.
\end{theorem}

Theorems~\ref{thm:chosen-spinor-package},
\ref{thm:kernel-nonfactor-split-line},
\ref{thm:levi-projective-exterior},
\ref{thm:square-det-levi}, and
\ref{thm:square-det-levi-image} give the chosen-model foundation, the
field-sensitive covering statements, the structural Levi-action bridge, and the
exact square-determinant Levi image theorem. The next section proves these
statements.

\section{Proofs of the main theorems}
\label{sec:highlights}

This section gives the split-model calculations, transport lemmas, and
orthogonal-action formulas used in Section~\ref{sec:results}.

Fix the split hyperbolic form $Q=\Hyp(W)$, choose a basis
$e_1,\dots,e_n$ of $W$ with dual basis $e^1,\dots,e^n$, and write
$e_I$ for the exterior basis vector indexed by a subset
$I\subseteq \{1,\dots,n\}$. Let $L_i(x)=e_i\wedge x$ and
$D_i(x)=\iota_{e^i}(x)$ be wedge and contraction. These operators satisfy the
split Clifford relations
\[
  L_i^2=D_i^2=0,
  \qquad
  L_iD_j+D_jL_i=\delta_{ij}.
\]

\begin{lemma}[Projector factors on the exterior basis]
For a basis vector $e_J$ and an index $i$, one has
\[
  L_iD_i(e_J)=
  \begin{cases}
    e_J,& i\in J,\\
    0,& i\notin J,
  \end{cases}
  \qquad
  D_iL_i(e_J)=
  \begin{cases}
    e_J,& i\notin J,\\
    0,& i\in J.
  \end{cases}
\]
Moreover, for distinct indices these operators commute on the exterior basis.
\end{lemma}

\begin{proof}
If $i\in J$, then $D_i(e_J)$ removes the basis vector $e_i$ from $e_J$ up to
the usual Koszul sign, and wedging again by $e_i$ restores $e_J$ with the same
sign. If $i\notin J$, then $D_i(e_J)=0$. This proves the formula for $L_iD_i$;
the formula for $D_iL_i$ is identical, with ``insert then remove'' replacing
``remove then insert.'' For distinct indices, the operators merely test
membership of different basis vectors, so applying them in either order to any
basis vector either reproduces that vector or kills it. Hence they commute on
the exterior basis, and therefore on all of $\bigwedge W$.
\end{proof}

\begin{theorem}[Matrix units in the split model]
\label{thm:matrix-units-split}
For the split hyperbolic form $Q=\Hyp(W)$, define for each subset
$I\subseteq \{1,\dots,n\}$
\[
  P_I=\prod_{i\in I} L_iD_i \prod_{j\notin I} D_jL_j .
\]
In this formula, and in the one below, all products are taken from left to
right in increasing index order. Then $P_I(e_J)=0$ for $J\neq I$ and
$P_I(e_I)=e_I$. For subsets $I,J$, let
\[
  T_{I,J}=
    \varepsilon_{I,J}
    \Bigl(\prod_{i\in I\setminus J} L_i\Bigr)
    \Bigl(\prod_{j\in J\setminus I} D_j\Bigr)
    P_J,
\]
where $\varepsilon_{I,J}\in\{\pm 1\}$ is the unique sign making
$T_{I,J}(e_J)=e_I$. Then
\[
  T_{I,J}(e_{J'})=\delta_{J,J'} e_I.
\]
Consequently the image of the split Clifford action contains every matrix unit
$E_{I,J}$ of $\End(\bigwedge W)$. Hence the split Clifford action is
surjective, and since both source and target have dimension $2^{2n}$, one obtains
\[
  \Cl(Q)\simeq \End(\bigwedge W).
\]
\end{theorem}

\begin{proof}
The preceding lemma shows that each factor in $P_I$ either fixes or kills a
basis vector according to whether the corresponding index belongs to that basis
vector. Since the factors commute on the exterior basis, the product $P_I$
annihilates every $e_J$ with $J\neq I$ and fixes $e_I$. Thus $P_I$ projects
onto the line $Ke_I$.

Now fix $J$. The operator $P_J$ kills every basis vector except $e_J$. On the
vector $e_J$, the contraction factors indexed by $J\setminus I$ are applied in
increasing order and remove exactly the basis elements that must disappear; the
result is the basis vector $e_{J\cap I}$ up to the corresponding Koszul sign.
The wedge factors indexed by $I\setminus J$, again in increasing order, insert
exactly the missing basis elements, producing $e_I$ up to a second Koszul sign.
Because $(J\setminus I)\cap(I\setminus J)=\varnothing$, none of these steps
kills the intermediate vector. With the chosen order, the total sign is the
single scalar $\varepsilon_{I,J}$ built into the definition of $T_{I,J}$.
Hence $T_{I,J}(e_{J'})=0$ for $J'\neq J$ and $T_{I,J}(e_J)=e_I$, so
$T_{I,J}=E_{I,J}$.

The image of the split Clifford action therefore contains all matrix units of
$\End(\bigwedge W)$ and is surjective. Finally,
$\dim \Cl(Q)=2^{\dim(W^*\oplus W)}=2^{2n}$ and
$\dim \End(\bigwedge W)=(2^n)^2=2^{2n}$, so the surjective map is an
isomorphism.
\end{proof}

\begin{corollary}[Even part and half-spin consequences]
\label{cor:even-half-spin}
For the same split hyperbolic form $Q=\Hyp(W)$, the even Clifford image
preserves the parity decomposition
$\bigwedge W=\wedgeEven\oplus\wedgeOdd$. On each parity block it contains all
matrix units between basis vectors of the same parity. If $\dim W>0$, then
\[
  \Cl^+(Q)\simeq \End(\wedgeEven)\times \End(\wedgeOdd).
\]
If $\dim W=0$, then $\Cl^+(Q)=K=\End(\wedgeEven)$ and $\wedgeOdd=0$. Moreover,
$\wedgeEven$ is simple as a $\Cl^+(Q)$-module, and if $\dim W>0$ then
$\wedgeOdd$ is simple as well. When $\dim W>0$ the two parity pieces are
inequivalent.
\end{corollary}

\begin{proof}
Each $L_i$ and $D_i$ changes degree by $\pm 1$, so even words preserve parity.
If $|I|\equiv |J|\pmod 2$, then the operator $T_{I,J}$ above contains an even
number of wedge/contraction steps and belongs to the even Clifford image.
Restricting to basis vectors of fixed parity therefore yields every block matrix
unit on $\wedgeEven$ and on $\wedgeOdd$. If $\dim W>0$, this gives a surjective
homomorphism
\[
  \Cl^+(Q)\twoheadrightarrow \End(\wedgeEven)\times \End(\wedgeOdd).
\]
Now
\[
  \begin{aligned}
  \dim \Cl^+(Q) &= 2^{2n-1}, \\
  \dim \End(\wedgeEven)+\dim \End(\wedgeOdd)
    &= 2(2^{n-1})^2 = 2^{2n-1}.
  \end{aligned}
\]
so the surjection is an isomorphism. If $\dim W=0$, then $\wedgeEven=K$,
$\wedgeOdd=0$, and the even Clifford algebra is \(K\).

The same blockwise matrix units show that any nonzero vector in $\wedgeEven$
generates the whole even half, and similarly for $\wedgeOdd$ when $n>0$.
When $n>0$, any $\Cl^+(Q)$-linear map $\wedgeEven\to\wedgeOdd$ is annihilated by
$0\times \End(\wedgeOdd)$ on the source and acted on faithfully by that factor
on the target, so it must vanish.
\end{proof}

\begin{proposition}[Transport to hyperbolic presentations and split rank]
\label{prop:transport-split}
If $e:Q \simeq \Hyp(W)$ is an explicit hyperbolic presentation, then
conjugating the split-model operators above by the linear identification
induced from $e$ transports the matrix-unit theorem and its parity corollary to
the Clifford action attached to $Q$. If $Q$ is finite-dimensional,
nondegenerate, and of split rank (equivalently maximal Witt index), a Witt decomposition provides
such an $e$, so the same proof gives the split-rank statements without keeping
the isometry as part of the final formulation.
\end{proposition}

\begin{proof}
Let $\rho_{\mathrm{hyp}}$ denote the split-model Clifford action on
$\bigwedge W$, and let
\[
  \phi_e:\Cl(Q)\xrightarrow{\ \simeq\ }\Cl(\Hyp(W))
\]
be the algebra isomorphism induced by the hyperbolic presentation $e$.
Transporting the action simply means defining
\[
  \rho_e(a)=\rho_{\mathrm{hyp}}(\phi_e(a))
  \qquad (a\in \Cl(Q)).
\]
For a vector $v\in V$, this says
\(\rho_e(\iota_Q(v)) = (\rho_{\mathrm{hyp}}\circ \iota_{\Hyp}\circ e)(v)\),
so the transported action is obtained by writing the split-model
wedge/contraction formulas in the coordinates supplied by $e$.
Because $\phi_e$ is an algebra isomorphism, the transported wedge/contraction
operators satisfy the same Clifford relations as the split ones. The
projectors $P_I$ and transfer operators $T_{I,J}$ are algebraic expressions in
those operators, so the entire matrix-unit construction survives unchanged
under transport. This gives the explicit-presentation version of both
Theorem~\ref{thm:matrix-units-split} and
Corollary~\ref{cor:even-half-spin}.

Now suppose $Q$ is finite-dimensional, nondegenerate, and of split rank
(equivalently maximal Witt index). Choose a maximal totally isotropic subspace $W\subset V$ and an
isotropic complement $W'$. Witt decomposition gives
$V=W'\oplus W$, and the bilinear pairing between $W'$ and $W$ identifies
$W'$ with $W^*$, producing a hyperbolic presentation
$Q\simeq \Hyp(W)$. Applying the preceding transport to that presentation
yields the split-rank statements. The final theorems depend only on the
existence of such a Witt decomposition, not on retaining a particular choice in
the statement.
\end{proof}

\paragraph{The spinor model versus the ambient regular module.}
The ambient regular module on $V$ has dimension
$2^{\dim V}$, whereas the chosen spinor model in split rank has the classical
dimension $2^{\dim V/2}$. The half-spin theorems therefore cannot literally be
statements about ambient submodules of the regular representation. Instead,
bridge results identify the canonical chosen-model halves with the even and odd
parts of the exterior model built on the maximal isotropic space:
$S^{\pm} = \bigwedge^{\mathrm{even}/\mathrm{odd}} W$.

\paragraph{Covering map and image calculations.}
The covering map $\Spin(V,Q) \to \SO(V,Q)$ is not treated in full
generality here. The results below establish the exact split-line image and the
exact square-determinant split-Levi image in finite split rank at least three.
All of these statements live under the standing \(2\in K^\times\) hypothesis. In the
split-rank setting treated here, if an
element lies in the kernel of the spin-to-isometry representation, then its
conjugation action is trivial on the Clifford generators, hence on the whole
Clifford algebra. The matrix model then forces such an element to be scalar,
and the spin norm reduces the scalar possibilities to $\pm1$.

\begin{proposition}[Kernel calculation from the matrix model]
\label{prop:kernel-from-matrix}
Assume $Q$ is finite-dimensional, nondegenerate, and hyperbolic. If
$x\in \Spin(Q)$ acts trivially on $V$ under the spin-to-isometry map, then
$x=\pm1$.
\end{proposition}

\begin{proof}
Triviality of the orthogonal action means $xvx^{-1}=v$ for every $v\in V$, so
$xv=vx$. Since $V$ generates $\Cl(Q)$, the element $x$ is central in
$\Cl(Q)$. Under the matrix-model isomorphism
$\Cl(Q)\simeq \End(\bigwedge W)$, central elements correspond to the center of a
full matrix algebra, hence to scalars. Thus $x=\lambda$ for some
$\lambda\in K^\times$. The defining spin norm relation then gives
$1=x\tilde x=\lambda^2$, so $\lambda=\pm1$. Transporting this argument through
the split-Witt reformulation gives the positive split-rank kernel theorem used in the
main results.
\end{proof}

\begin{proposition}[Positive split-rank non-factorization]
\label{prop:nonfactorization}
Assume $Q$ is finite-dimensional, nondegenerate, and of positive split rank,
and let $S=\bigwedge W$ be the exterior-model spinor module attached to a
chosen Witt decomposition of $Q$. Then the spin representation
\[
  \Spin(V,Q)\to \operatorname{GL}(S)
\]
does not factor through $\SO(V,Q)$. The same is true for the induced
representations on the nonzero half-spin modules $S^+$ and $S^-$.
\end{proposition}

\begin{proof}
Choose a nonzero $w\in W$ and a dual vector $f\in W^*$ with $f(w)=1$. In the
transported hyperbolic model, the vector
\[
  u=(-f,w)\in W^*\oplus W
\]
satisfies
\[
  Q(u)=(-f)(w)=-1,
\]
hence $u^2=-1$. Therefore
\[
  -1=u^2
\]
is a product of two \(Q=-1\) vector generators, so $-1\in \Spin(V,Q)$. Being
scalar, it acts trivially on $V$ by conjugation, so its image in $\SO(V,Q)$ is
the identity.
On the spinor module $S=\bigwedge W$, however, the Clifford action of the
scalar $-1$ is $-\mathrm{id}_S$, which is nontrivial because
$(-1)\cdot 1=-1\neq 1$ in $S$. Since scalar multiplication preserves parity,
the same operator restricts to $-\mathrm{id}$ on each half-spin piece. The
positive half contains $1$, and the negative half is nonzero because
$w\in W\subset \bigwedge^1 W$. Thus the action of $-1$ is nontrivial on $S$,
$S^+$, and $S^-$. Therefore none of these representations can factor through
$\SO(V,Q)$.
\end{proof}

\begin{proposition}[Positive split-rank projective descent]
\label{prop:projective-descent}
Assume $Q$ is finite-dimensional, nondegenerate, and of positive split rank,
and let $S=\bigwedge W$ be the exterior-model spinor module attached to a
chosen Witt decomposition of $Q$. Then the induced action of $\Spin(V,Q)$ on
the projective space $\mathbb{P}(S)$ of one-dimensional subspaces depends only
on the image of $\Spin(V,Q)\to\SO(V,Q)$ and therefore defines an action of
that image subgroup. The same is true for the projective spaces of the nonzero
half-spin modules $S^+$ and $S^-$.
\end{proposition}

\begin{proof}
Suppose $x,y\in \Spin(V,Q)$ have the same image in $\SO(V,Q)$. Then
$y^{-1}x$ acts trivially on $V$, so by the kernel calculation above one has
$y^{-1}x=\pm1$. On $S$, and likewise on each half-spin piece, the scalars
$\pm1$ act by scalar multiplication and therefore fix every one-dimensional
subspace. Thus $x$ and $y$ induce the same transformation on
$\mathbb{P}(S)$, $\mathbb{P}(S^+)$, and $\mathbb{P}(S^-)$. Hence these
projective actions depend only on the image in $\SO(V,Q)$ and therefore define
actions of the image subgroup of $\Spin(V,Q)\to\SO(V,Q)$.
\end{proof}

\begin{proposition}[Explicit hyperbolic transvection lift]
\label{prop:transvection-unit}
Let $Q=\Hyp(W)$. If $\delta\in W^*$ and $w\in W$ satisfy $\delta(w)=0$,
set \(a=(\delta,0)\), \(b=(0,w)\), and
\(x_{\delta,w}=1+\iota(a)\iota(b)\in \Cl(Q)\).
Then $x_{\delta,w}$ is an even Clifford unit with
\[
  x_{\delta,w}^{-1}=1-\iota(a)\iota(b),
\]
and for every $(d,u)\in W^*\oplus W$ one has
\[
  x_{\delta,w}\,\iota(d,u)\,x_{\delta,w}^{-1}
  =
  \iota\bigl(d+d(w)\delta,\ u-\delta(u)w\bigr).
\]
Equivalently, conjugation by $x_{\delta,w}$ realizes the transported hyperbolic
transvection attached to $(\delta,w)$.
\end{proposition}

\begin{proof}
Write $n=\iota(a)\iota(b)$. Because $Q(a)=Q(b)=0$ and
$\langle a,b\rangle=\delta(w)=0$, one has $n^2=0$. Hence
\[
  (1+n)(1-n)=1=(1-n)(1+n),
\]
so $x_{\delta,w}^{-1}=1-n$. The element $n$ is a product of two vectors, hence
lies in the even Clifford part, and therefore so does $x_{\delta,w}$.

Now let $z=(d,u)$. Using the Clifford anticommutation relation in the
hyperbolic form,
\[
  \iota(b)\iota(z)+\iota(z)\iota(b)=\langle b,z\rangle=d(w),
  \qquad
  \iota(z)\iota(a)+\iota(a)\iota(z)=\langle z,a\rangle=\delta(u),
\]
one obtains
\[
  n\,\iota(z)=d(w)\iota(a)-\iota(a)\iota(z)\iota(b),
  \qquad
  \iota(z)\,n=\delta(u)\iota(b)-\iota(a)\iota(z)\iota(b).
\]
Since $n^2=0$, also $n\,\iota(z)\,n=0$, and therefore
\[
  \begin{aligned}
  x_{\delta,w}\,\iota(z)\,x_{\delta,w}^{-1}
    &= (1+n)\iota(z)(1-n) \\
    &= \iota(z)+\bigl(n\iota(z)-\iota(z)n\bigr) \\
    &= \iota\bigl(d+d(w)\delta,\ u-\delta(u)w\bigr).
  \end{aligned}
\]
This is the displayed transvection formula.
\end{proof}

\paragraph{Unipotent and semisimple directions in the hyperbolic model.}
Proposition~\ref{prop:transvection-unit} gives an explicit Clifford lift of the
unipotent root-subgroup direction in the hyperbolic Levi picture. The next
proposition provides the complementary semisimple direction along a chosen split
line.

\begin{proposition}[Chosen-line square scaling in the hyperbolic model]
\label{prop:chosen-line-square}
Let $Q=\Hyp(W)$. Choose $w\in W$ and $f\in W^*$ with $f(w)=1$. For
$t\in K^\times$, define
\[
  \lambda_t(d,u)=\bigl(d+(t^{-2}-1)d(w)f,\ u+(t^2-1)f(u)w\bigr).
\]
Then $\lambda_t$ scales the line $Kw$ by $t^2$, scales the dual line $Kf$ by
$t^{-2}$, fixes $\ker(f)\subset W$ and $\ker(\operatorname{ev}_w)\subset W^*$,
and lies in the image of $\Spin(Q)\to \SO(Q)$. More precisely, if
\[
  u_t=(-t^{-1}f,tw),
  \qquad
  v=(-f,w),
\]
then $Q(u_t)=Q(v)=-1$ and the spin element $u_t v$ acts on $W^*\oplus W$ as
$\lambda_t$.
\end{proposition}

\begin{proof}
Because $f(w)=1$,
\[
  Q(u_t)=(-t^{-1}f)(tw)=-1,
  \qquad
  Q(v)=(-f)(w)=-1,
\]
so $u_t$ and $v$ are vectors of square $-1$, hence $u_t v\in \Spin(Q)$. A
direct reflection calculation gives, for any $c\in K^\times$,
\[
  r_{(-c^{-1}f,cw)}(d,u)
  =\bigl((d(w)-c^{-2}f(u))f-d,\ -(c^2d(w)-f(u))w-u\bigr).
\]
Applying first the reflection with $c=1$ and then the reflection with $c=t$
yields
\[
  (d,u)\longmapsto
  \bigl(d+(t^{-2}-1)d(w)f,\ u+(t^2-1)f(u)w\bigr)=\lambda_t(d,u).
\]
The displayed formula makes the scaling and fixed-kernel claims immediate, so
$\lambda_t$ is a square scaling along the chosen split line and lies in the
spin image.
\end{proof}

If $\ell_t\in \operatorname{GL}(W)$ denotes the underlying line scaling
\[
  \ell_t(u)=u+(t^2-1)f(u)w,
\]
then the same explicit lift has a fully normalized chosen-model action:
\[
  \rho_S(u_t v)=-t^{-1}\,(\bigwedge \ell_t).
\]
The minus sign is part of the normalization, not an error: the lift is a product
of two square-$-1$ vectors, and the chosen exterior model records this
two-reflection torus element with the scalar $-t^{-1}$ while its orthogonal
action scales the chosen line by $t^2$.
This is the exact torus-lift refinement of the projective Levi-action theorem.

\begin{corollary}[Chosen-line torus control of the transvection parameter]
\label{cor:torus-transvection}
Fix $w\in W$ and $f\in W^*$ with $f(w)=1$. Let $\lambda_t$ be the square-scaling
map from Proposition~\ref{prop:chosen-line-square}, and for $\delta(w)=0$ let
$T_{\delta,w}$ denote the transvection from
Proposition~\ref{prop:transvection-unit},
\[
  T_{\delta,w}(d,u)=\bigl(d+d(w)\delta,\ u-\delta(u)w\bigr).
\]
Then
\[
  \lambda_t\,T_{\delta,w}\,\lambda_t^{-1}=T_{t^2\delta,w}.
\]
Consequently the explicit lifts in
Propositions~\ref{prop:transvection-unit} and~\ref{prop:chosen-line-square}
control both the unipotent root subgroup
\[
  U_w=\{\,T_{\delta,w} : \delta(w)=0\,\}
\]
and the square torus generated by the $\lambda_t$, with the weight-$2$
conjugation action.
\end{corollary}

\begin{proof}
Write $d=d_0+d(w)f$ with $d_0(w)=0$ and $u=u_0+f(u)w$ with $f(u_0)=0$. Then
\[
  \lambda_t(d,u)=\bigl(d_0+t^{-2}d(w)f,\ u_0+t^2f(u)w\bigr),
\]
so $\lambda_t^{-1}$ is obtained by replacing $t$ with $t^{-1}$. If
$\delta(w)=0$, then $\lambda_t$ fixes $\delta$ and scales $w$ by $t^2$. Applying
$\lambda_t^{-1}$, then $T_{\delta,w}$, and then $\lambda_t$ gives
\[
  (d,u)\longmapsto \bigl(d+t^2d(w)\delta,\ u-t^2\delta(u)w\bigr)
  = T_{t^2\delta,w}(d,u),
\]
which is the displayed conjugation formula. The final claim is the corresponding
subgroup interpretation: the formulas give explicit Clifford lifts of both the
unipotent subgroup $U_w$ and the semisimple square torus acting on it.
\end{proof}

\begin{corollary}[Internal Clifford-level torus action]
\label{cor:internal-clifford-torus}
Fix $w\in W$ and $f\in W^*$ with $f(w)=1$, and let
\[
  s_t = \iota(-t^{-1}f,tw)\,\iota(-f,w)\in \Spin(\Hyp(W))
\]
be the explicit chosen-line square-scaling lift from
Proposition~\ref{prop:chosen-line-square}. For $\delta(w)=0$, the explicit
transvection units satisfy
\[
  s_t\,x_{\delta,w}\,s_t^{-1}=x_{t^2\delta,w}.
\]
Thus the weight-$2$ torus action already holds inside the explicit Clifford
representatives, not only after passing to the orthogonal action.
\end{corollary}

\begin{proof}
By Proposition~\ref{prop:chosen-line-square}, conjugation by $s_t$ fixes the
vector $(\delta,0)$ and sends $(0,w)$ to $(0,t^2w)$. Therefore
\[
  s_t\,x_{\delta,w}\,s_t^{-1}
  = 1 + \iota(\delta,0)\,\iota(0,t^2w).
\]
Because $\iota$ is $K$-linear and scalars are central in the Clifford algebra,
\[
  \iota(\delta,0)\,\iota(0,t^2w)
  = \iota(\delta,0)\,t^2\iota(0,w)
  = \iota(t^2\delta,0)\,\iota(0,w),
\]
so the right-hand side is $x_{t^2\delta,w}$.
\end{proof}

\begin{corollary}[Semidirect product along a chosen split line]
\label{cor:one-line-semidirect}
For fixed $w\in W$ and $f\in W^*$ with $f(w)=1$, the maps $T_{\delta,w}$ with
$\delta(w)=0$ satisfy
\[
  T_{\delta,w}\,T_{\eta,w}=T_{\delta+\eta,w}
  \qquad
  (\delta(w)=\eta(w)=0),
\]
so they form an additive subgroup naturally identified with
$\ker(\operatorname{ev}_w)\subset W^*$. Together with the square scalings
$\lambda_t$, these maps form the semidirect product
\[
  \ker(\operatorname{ev}_w)\rtimes K^\times,
  \qquad
  t\cdot \delta = t^2\delta,
\]
inside the orthogonal action along the chosen line $Kw$.
\end{corollary}

\begin{proof}
Apply the explicit formula for $T_{\eta,w}$ first:
\[
  (d,u)\longmapsto \bigl(d+d(w)\eta,\ u-\eta(u)w\bigr).
\]
Because $\eta(w)=0$, the scalar $d(w)$ is unchanged on the dual side, and
$\delta(u-\eta(u)w)=\delta(u)$ because $\delta(w)=0$. Therefore
\[
  T_{\delta,w}T_{\eta,w}(d,u)
  =
  \bigl(d+d(w)(\delta+\eta),\ u-(\delta(u)+\eta(u))w\bigr)
  =
  T_{\delta+\eta,w}(d,u).
\]
This proves the additive subgroup law. The semidirect-product statement is then
the combination of this identity with
Corollary~\ref{cor:torus-transvection}, which shows
$\lambda_t T_{\delta,w}\lambda_t^{-1}=T_{t^2\delta,w}$.
\end{proof}

The previous results supply the explicit root and torus directions. We next
prove the structural Levi-action theorem from
Section~\ref{sec:results}.

\begin{proof}[Proof of Theorem~\ref{thm:levi-projective-exterior}]
Let $v=\rho_S(s)1\in S$. For any $d\in W^*$, because $s$ projects to
$\Lambda(g)$, conjugation carries the dual generator $\iota(g^{\vee}d,0)$ to
$\iota(d,0)$. Applying both sides to $1$ gives
\[
  \iota(d,0)\,v
  =
  \rho_S(s)\bigl(\iota(g^{\vee}d,0)\cdot 1\bigr)
  =0,
\]
since dual generators act on $S$ by contraction and annihilate the vacuum
vector $1$. Thus every contraction operator kills $v$, and the vacuum-line
criterion from Section~\ref{sec:arch} gives $v=c\cdot 1$ for some $c\in K$.

Similarly, for every $w\in W$, conjugation carries $\iota(0,w)$ to
$\iota(0,gw)$, so $\rho_S(s)$ intertwines exterior multiplication by $w$ with
exterior multiplication by $gw$. Starting from $\rho_S(s)1=c$ and inducting on
exterior monomials yields
\[
  \rho_S(s)=c\,(\bigwedge g).
\]
Because both $\rho_S(s)$ and $\bigwedge g$ are invertible, necessarily
$c\ne 0$.
\end{proof}

The preceding theorem identifies the chosen-model action of any split Levi lift
up to scalar. We next prove the corresponding higher-rank factorization and
explicit-lift theorem from Section~\ref{sec:results}.

\begin{proposition}[Auxiliary-index commutator for elementary Levi transvections]
\label{prop:auxiliary-index-commutator}
Let \(Q=\Hyp(W)\), choose a basis \((e_\ell)_{\ell\in I}\) of \(W\), and
write \(\varepsilon_\ell\) for the dual coordinate functionals. If \(i,j,k\in I\)
are pairwise distinct and \(c\in K\), then the transported elementary Levi
transvection
\[
  (d,u)\longmapsto
  \bigl(d-c\,d(e_i)\varepsilon_j,\ u+c\,\varepsilon_j(u)e_i\bigr)
\]
lies in the image of \(\Spin(Q)\to\SO(Q)\).
\end{proposition}

\begin{proof}
For distinct \(p,q\in I\) and \(a\in K\), set
\[
  s_{pq}(a)=
    \iota\bigl(-(\varepsilon_p+a\varepsilon_q),e_p\bigr)
    \iota(-\varepsilon_p,e_p)\in \Spin(Q).
\]
Both vectors have \(Q\)-value \(-1\), since
\((\varepsilon_p+a\varepsilon_q)(e_p)=1\). Let \(P_{pq}(a)\) be the image of
\(s_{pq}(a)\) in \(\SO(Q)\). The two-reflection formula gives, for
\((d,u)\in W^*\oplus W\),
\[
\begin{aligned}
P_{pq}(a)(d,u)=
\bigl(&d+(d(e_p)-\varepsilon_p(u))a\varepsilon_q
        +a\varepsilon_q(u)(\varepsilon_p+a\varepsilon_q),\\
      &u-a\varepsilon_q(u)e_p\bigr).
\end{aligned}
\]
Put \(x_\ell=d(e_\ell)\) and \(y_\ell=\varepsilon_\ell(u)\). In the
\((p,q)\)-coordinates this formula is exactly
\[
  x_p\mapsto x_p+a y_q,\qquad
  x_q\mapsto x_q+a(x_p-y_p)+a^2y_q,\qquad
  y_p\mapsto y_p-a y_q,
\]
with all other displayed \(x\)- and \(y\)-coordinates unchanged except those
forced by these three assignments.

The strategy is an elementary-commutator identity with an auxiliary index:
the \(k\)-supported pair generators introduce and then cancel the
auxiliary Clifford root terms, leaving only the
Levi transvection in the \((i,j)\) root direction. We apply the coordinate rule to
the product
\[
  P_{ki}(1)\,P_{kj}(c/2)\,P_{ki}(-1)\,P_{kj}(-c/2)\,P_{ij}(-c),
\]
where the rightmost factor acts first. Writing only the six coordinates with
indices \(i,j,k\), the successive states are
\begingroup
\small
\[
\begin{aligned}
z_0={}&(x_i,\ x_j,\ x_k;\ y_i,\ y_j,\ y_k),\\
z_1={}&(x_i-cy_j,\ x_j-cx_i+cy_i+c^2y_j,\ x_k;\\
&\quad y_i+cy_j,\ y_j,\ y_k),\\
z_2={}&(x_i-cy_j,
  x_j-cx_i+cy_i-\tfrac c2x_k+\tfrac c2y_k+\tfrac54c^2y_j,
  x_k-\tfrac c2y_j;\\
&\quad y_i+cy_j,\ y_j,\ y_k+\tfrac c2y_j),\\
z_3={}&(x_i+cy_j-x_k+y_i+y_k,
  x_j-cx_i+cy_i-\tfrac c2x_k+\tfrac c2y_k+\tfrac54c^2y_j,
  x_k-y_i-\tfrac32cy_j;\\
&\quad y_i+cy_j,\ y_j,\ y_k+y_i+\tfrac32cy_j),\\
z_4={}&(x_i+cy_j-x_k+y_i+y_k,\ x_j-cx_i,\ x_k-y_i-cy_j;\\
&\quad y_i+cy_j,\ y_j,\ y_k+y_i+cy_j),\\
z_5={}&(x_i,\ x_j-cx_i,\ x_k;\ y_i+cy_j,\ y_j,\ y_k).
\end{aligned}
\]
\endgroup
after applying \(P_{ij}(-c)\), \(P_{kj}(-c/2)\), \(P_{ki}(-1)\),
\(P_{kj}(c/2)\), and \(P_{ki}(1)\), respectively. In the final step, the
auxiliary coordinates return to their original values.
Thus the product sends
\[
 (x_i,x_j,x_k;\,y_i,y_j,y_k)
 \longmapsto
 (x_i,\ x_j-cx_i,\ x_k;\, y_i+cy_j,\ y_j,\ y_k).
\]
All coordinates outside \(\{i,j,k\}\) are fixed throughout. Translating this
coordinate statement back to \(W^*\oplus W\) gives exactly
\[
  (d,u)\longmapsto
  \bigl(d-c\,d(e_i)\varepsilon_j,\ u+c\,\varepsilon_j(u)e_i\bigr).
\]
Since each \(P_{pq}(a)\) is the image of the spin element \(s_{pq}(a)\), their
product is also in the spin image. Hence the transported elementary Levi
transvection is in the spin image.
\end{proof}

\begin{proof}[Proof of Theorem~\ref{thm:square-det-levi}]
Choose a basis \(e_1,\ldots,e_n\) of $W$ and write the matrix of $g$ in that
basis. The classical transvection reduction for invertible matrices in dimension
at least \(2\), equivalently Gaussian elimination by elementary matrices, writes
\[
  g=A\,D\,B,
\]
where \(A\) and \(B\) are products of elementary basis transvections and
\(D=\operatorname{diag}(t_1,\ldots,t_n)\). Since \(\det(g)=u^2\), we have
\(\prod_i t_i=u^2\). Fix the distinguished line \(Ke_1\). Then
\[
  D=L_1(u^2)\prod_{j=2}^n D_{j1}(t_j),
\]
where \(L_1(u^2)\) scales \(e_1\) by \(u^2\) and fixes the other basis vectors,
and \(D_{j1}(a)\) scales \(e_j\) by \(a\), scales \(e_1\) by \(a^{-1}\), and
fixes the remaining basis vectors. Indeed, the right side scales \(e_j\) by
\(t_j\) for \(j>1\), while it scales \(e_1\) by
\[
  u^2\prod_{j=2}^n t_j^{-1}=t_1.
\]

Each determinant-one two-line block is elementary. With \(T_{ij}(r)\) denoting
the transvection \(e_j\mapsto e_j+r e_i\), direct multiplication on the
\((e_i,e_j)\)-plane gives
\[
  D_{ij}(a)=
  T_{ij}(a-1)\,T_{ji}(1)\,T_{ij}(a^{-1}-1)\,T_{ji}(-a).
\]
Transporting the decomposition of \(g\) through the Levi embedding
\(\Lambda(h)=(h^{-{\vee}},h)\) therefore expresses \(\Lambda(g)\) as transported
hyperbolic transvections, one chosen-line square scaling, and more transported
hyperbolic transvections. This is the asserted orthogonal factorization.

It remains to compute the action on the chosen spinor model. For a transported
hyperbolic transvection with \(\delta(w)=0\), Proposition~\ref{prop:transvection-unit}
uses the Clifford unit
\[
  x_{\delta,w}=1+\iota(\delta,0)\iota(0,w).
\]
On \(S=\bigwedge W\) this operator is
\[
  1+(\delta\wedge -)\,\iota_w,
\]
which is exactly \(\bigwedge(1+\delta\otimes w)\), because
\((\delta\otimes w)^2=0\). Thus every transported transvection lift acts with
scalar \(1\) on the exterior model. The chosen-line square scaling
from Proposition~\ref{prop:chosen-line-square}, applied with parameter \(u\),
acts as \(-u^{-1}\bigwedge L_1(u^2)\). Multiplying the units attached to the
factorization above therefore gives an even unitary Clifford unit \(x\) with
\[
  \rho_S(x)=-u^{-1}(\bigwedge g),
\]
which is the claimed explicit chosen-model lift. If \(g\in\operatorname{SL}(W)\)
then \(u=1\), and the chosen-line factor is \(L_1(1)=1\); hence
\(\Lambda(g)\) is generated by transported hyperbolic transvections alone.
\end{proof}

\begin{proposition}[One-line spin-image obstruction]
\label{prop:line-projector-converse}
Let \(Q=\Hyp(W)\), choose a basis of \(W\), and fix a basis index \(i_0\).
For \(t\in K^\times\), let \(L_{i_0}(t)\) scale \(e_{i_0}\) by \(t\) and fix all
other basis vectors. If
\[
  \Lambda(L_{i_0}(t))\in
  \operatorname{im}\bigl(\Spin(Q)\to\SO(Q)\bigr),
\]
then \(t\in(K^\times)^2\).
\end{proposition}

\begin{proof}
Let \(s\in\Spin(Q)\) lift \(\Lambda(L_{i_0}(t))\). By
Theorem~\ref{thm:levi-projective-exterior}, there exists \(c\in K^\times\) such
that
\[
  \rho_S(s)=c\,\bigwedge L_{i_0}(t)
\]
on \(S=\bigwedge W\). Let \(P_{\mathrm{in}}\) and \(P_{\mathrm{out}}\) be the
complementary occupation projectors associated with \(e_{i_0}\):
\[
  P_{\mathrm{in}}=\iota(0,e_{i_0})\iota(e_{i_0}^*,0),\qquad
  P_{\mathrm{out}}=\iota(e_{i_0}^*,0)\iota(0,e_{i_0}).
\]
Under the split exterior action, \(P_{\mathrm{in}}\) projects onto the span of
the exterior monomials containing \(e_{i_0}\), and \(P_{\mathrm{out}}\) projects
onto the span of the monomials not containing \(e_{i_0}\). Thus
\[
  P_{\mathrm{in}}+P_{\mathrm{out}}=1,\qquad
  P_{\mathrm{in}}P_{\mathrm{out}}=P_{\mathrm{out}}P_{\mathrm{in}}=0.
\]
Clifford conjugation interchanges them:
\[
  \widetilde{P_{\mathrm{in}}}=P_{\mathrm{out}},\qquad
  \widetilde{P_{\mathrm{out}}}=P_{\mathrm{in}}.
\]
The displayed formula for \(\rho_S(s)\) gives
\[
  sP_{\mathrm{in}}=(ct)P_{\mathrm{in}},
  \qquad
  sP_{\mathrm{out}}=cP_{\mathrm{out}},
\]
as Clifford elements, because the split Clifford action is faithful. Hence
\[
  s=s(P_{\mathrm{in}}+P_{\mathrm{out}})
   =ctP_{\mathrm{in}}+cP_{\mathrm{out}},
\]
and applying Clifford conjugation gives
\[
  \tilde{s}=ctP_{\mathrm{out}}+cP_{\mathrm{in}}.
\]
Since \(s\in\Spin(Q)\), one has \(s\tilde{s}=1\). Therefore
\[
\begin{aligned}
  1=s\tilde{s}
    &=(ctP_{\mathrm{in}}+cP_{\mathrm{out}})
      (ctP_{\mathrm{out}}+cP_{\mathrm{in}})\\
    &=c^2t(P_{\mathrm{in}}+P_{\mathrm{out}})
      =c^2t.
\end{aligned}
\]
Thus \(c^2t=1\), and hence \(t=(c^{-1})^2\), as required.
\end{proof}

\begin{proof}[Proof of Theorem~\ref{thm:square-det-levi-image}]
Let \(H\) be the image of \(\Spin(\Hyp(W))\to\SO(\Hyp(W))\). We first
prove the implication from square determinant to spin-image membership. Since
\(H\) is a subgroup, it is enough to prove that the standard generators of the
square-determinant Levi subgroup lie in \(H\). Relative to a basis, this subgroup
is generated by transported elementary basis transvections and by basis-diagonal
scalings whose total determinant is a square. Classically this same subgroup is
the kernel of the Levi spinor norm, because the spinor norm on the split Levi is
\(\det\) modulo \((K^\times)^2\); the argument below uses only the displayed
square-determinant generator description.

First consider an elementary transported basis transvection in the \((i,j)\)
root direction. Because \(\dim W\ge3\), there is a basis index \(k\) distinct
from both \(i\) and \(j\). Proposition~\ref{prop:auxiliary-index-commutator}
then writes that transvection as a product of five pair generators
\[
  P_{ki}(1)\,P_{kj}(c/2)\,P_{ki}(-1)\,P_{kj}(-c/2)\,P_{ij}(-c),
\]
each of which is the image of an explicit spin element. Thus every transported
elementary basis transvection lies in \(H\).

It remains to treat a basis-diagonal scaling
\[
  D=\operatorname{diag}(t_\ell)_{\ell\in I},\qquad \prod_{\ell\in I}t_\ell=u^2.
\]
Choose a distinguished index \(i_0\). The square-determinant diagonal
factorization gives
\[
  D=L_{i_0}(u^2)\prod_{j\ne i_0}D_{j i_0}(t_j),
\]
where \(L_{i_0}(u^2)\) is the chosen-line square scaling on \(Ke_{i_0}\), and
\(D_{ij}(a)\) denotes the determinant-one block scaling \(e_i\) by \(a\) and
\(e_j\) by \(a^{-1}\). Proposition~\ref{prop:chosen-line-square} puts
\(L_{i_0}(u^2)\) in \(H\). Each two-line block admits the four-transvection
factorization
\[
\begin{aligned}
  D_{ij}(a)=&
  T_{ij}(a-1)\,T_{ji}(1)\,T_{ij}(a^{-1}-1)\,T_{ji}(-a),
\end{aligned}
\]
with \(T_{ij}(r)\) denoting the elementary transvection
\(e_j\mapsto e_j+r e_i\) and fixing the other basis vectors. The transvection
case above places every factor on the right in \(H\), hence
\(D_{ij}(a)\in H\), and therefore the whole diagonal scaling \(D\) lies in
\(H\).

Both types of generators of the square-determinant Levi subgroup are therefore
in \(H\). Since \(H\) is a subgroup, the subgroup generated by them is contained
in \(H\). This proves
\[
  \det(g)\in(K^\times)^2 \quad\Longrightarrow\quad \Lambda(g)\in H.
\]

For the converse, suppose \(\Lambda(g)\in H\), and put
\(d=\det(g)\). Choose a basis index \(i_0\), and let \(L_{i_0}(d)\) be the
diagonal linear automorphism of \(W\) that multiplies \(e_{i_0}\) by \(d\) and
fixes the other basis vectors. Then
\[
  q:=L_{i_0}(d)^{-1}g
\]
has determinant \(1\). By the forward implication,
\(\Lambda(q)\in H\). Hence
\[
  \Lambda(L_{i_0}(d))
    =\Lambda(g)\Lambda(q)^{-1}
\]
also lies in \(H\). Proposition~\ref{prop:line-projector-converse} applied with
\(t=d\) gives \(d\in(K^\times)^2\). This proves the converse and hence the exact
criterion.
\end{proof}

\paragraph{Example: nonsquare-field obstruction.}
Over \(K=\mathbb{Q}\), the split-line reciprocal scaling
\[
  (x,y)\longmapsto (2x,2^{-1}y)
\]
lies in \(\SO(\Hyp_{\mathbb{Q}}(\mathbb{Q}))\) but not in the spin image,
because \(2\notin(\mathbb{Q}^\times)^2\). In rank at least three the same
obstruction appears on the split Levi: the diagonal \(g=\operatorname{diag}
(2,1,\ldots,1)\) gives \(\det(g)=2\), hence \(\Lambda(g)\) is an orthogonal
Levi element but is not in the image of
\(\Spin(\Hyp_{\mathbb{Q}}(W))\to\SO(\Hyp_{\mathbb{Q}}(W))\).

\paragraph{A rank-$2$ shear.}
If $W=K^2$ with basis $e_1,e_2$, take $w=e_1$, $f=e_1^*$, and
$\delta=e_2^*$. Then Proposition~\ref{prop:transvection-unit} gives the
explicit higher-rank shear
\[
  (d_1,d_2,u_1,u_2)\longmapsto (d_1,d_2+d_1,\ u_1-u_2,\ u_2).
\]
Corollary~\ref{cor:torus-transvection} says that the chosen-line square
scalings rescale this shear parameter by $t^2$, exhibiting the standard
root-group / torus interaction already inside the explicit Clifford formulas.

\begin{proposition}[Split-line square-scaling calculation]
\label{prop:split-line-square}
Let $Q=\Hyp(K)$, and choose isotropic generators $p$ and $q$ with
\(p^2=q^2=0\) and \(pq+qp=1\). Every even Clifford element can then be written
as \(x=a\,pq+b\,qp\). If \(x\in \Spin(Q)\), the defining norm condition forces
\(ab=1\), and an explicit conjugation calculation gives
\[
  x p x^{-1}=a^2 p,
  \qquad
  x q x^{-1}=a^{-2} q.
\]
Hence the image of $\Spin(Q)\to \SO(Q)$ is exactly the square-scaling subgroup.
\end{proposition}

\begin{proof}
Because the even part is spanned by $pq$ and $qp$, every even element has the
stated form. Using the relations $p^2=q^2=0$ and $pq+qp=1$, one has
\[
  (pq)^2=pq,\qquad (qp)^2=qp,\qquad pq\,qp=qp\,pq=0,\qquad \widetilde{pq}=qp.
\]
Hence
\[
  \tilde x = a\,qp+b\,pq,
\]
so
\[
  x\tilde x
  =(a\,pq+b\,qp)(a\,qp+b\,pq)
  =ab(pq+qp)
  =ab.
\]
Thus the spin condition $x\tilde x=1$ is equivalent to $ab=1$, and in that case
$x^{-1}=\tilde x$. Moreover,
\[
  xp=(a\,pq+b\,qp)p=a\,p,\qquad xq=(a\,pq+b\,qp)q=b\,q,
\]
because $pqp=p$, $qpp=0$, $pqq=0$, and $qpq=q$. Therefore
\[
  xpx^{-1}=a\,p(a\,qp+b\,pq)=a^2p,
  \qquad
  xqx^{-1}=b\,q(a\,qp+b\,pq)=b^2q=a^{-2}q.
\]
Hence the induced orthogonal action rescales the two isotropic lines by
reciprocal squares, so the image is contained in the square-scaling subgroup.
Conversely, Proposition~\ref{prop:chosen-line-square} specializes to the split
line and shows that every square-scaling occurs in the image. Thus the image is
the square-scaling subgroup.
\end{proof}

Thus the spin-to-orthogonal map on the split line is not surjective over fields
with nonsquare units. If the square map on $K^\times$ is surjective, then one
recovers the full split-line double cover. Combined with the positive split-rank
linear/projective descent picture, it shows that the exterior-model spin action
has a sharp two-level structure over a general field: projectively it depends
only on the orthogonal image in every positive split rank, while linearly it
remembers the scalar kernel element $-1$; in rank $1$, the image can
additionally miss nonsquare reciprocal scalings.

\section{Conclusion}
\label{sec:conclusion}

The classical exterior-model construction supplies the representation space
\(\bigwedge W\). The calculations above determine how the spin group maps
into the split Levi part of the orthogonal group over an arbitrary field
\(K\) with \(2\in K^\times\). The central structural result is that every split
Levi lift acts projectively as the natural exterior action. Explicit
transvection and square-scaling Clifford representatives yield the forward
square-determinant inclusion, while the line-projector argument and spin
unitarity supply the converse. In finite split rank at least three,
the spin image on the split Levi is therefore exactly the square-determinant
subgroup.

\section*{Statements and Declarations}

\paragraph{Competing Interests.}
The authors have no competing interests to declare that are relevant to the
content of this article.

\paragraph{Data Availability.}
The formal verification source code supporting the results of this article is
publicly available at \url{https://github.com/Arthur742Ramos/SpinorLean}.
No other datasets were generated or analysed during the current study.

\bibliographystyle{spmpsci}
\bibliography{refs}

@book{LawsonMichelsohn1989,
  author    = {Lawson, Jr., H. Blaine and Michelsohn, Marie-Louise},
  title     = {Spin Geometry},
  publisher = {Princeton University Press},
  series    = {Princeton Mathematical Series},
  volume    = {38},
  year      = {1989},
  address   = {Princeton, NJ},
  isbn      = {978-0-691-08542-5}
}

@book{Chevalley1997,
  author    = {Chevalley, Claude},
  title     = {The Algebraic Theory of Spinors and Clifford Algebras},
  publisher = {Springer},
  series    = {Collected Works},
  volume    = {2},
  year      = {1997},
  note      = {Reprint of the 1954 original with a postface by J.-P. Bourguignon}
}

@article{AtiyahBottShapiro1964,
  author  = {Atiyah, Michael F. and Bott, Raoul and Shapiro, Arnold},
  title   = {Clifford Modules},
  journal = {Topology},
  volume  = {3},
  number  = {Suppl. 1},
  pages   = {3--38},
  year    = {1964},
  doi     = {10.1016/0040-9383(64)90003-5}
}

@book{Lounesto2001,
  author    = {Lounesto, Pertti},
  title     = {Clifford Algebras and Spinors},
  edition   = {2nd},
  publisher = {Cambridge University Press},
  series    = {London Mathematical Society Lecture Note Series},
  volume    = {286},
  year      = {2001},
  address   = {Cambridge},
  isbn      = {978-0-521-80579-5}
}

@book{Artin1957GeometricAlgebra,
  author    = {Artin, Emil},
  title     = {Geometric Algebra},
  publisher = {Interscience Publishers},
  address   = {New York},
  year      = {1957},
  note      = {Reprinted by Wiley Classics Library, 1988}
}

@book{OMeara2000QuadraticForms,
  author    = {O'Meara, O. Timothy},
  title     = {Introduction to Quadratic Forms},
  publisher = {Springer},
  series    = {Classics in Mathematics},
  address   = {Berlin},
  year      = {2000},
  note      = {Reprint of the 1963 edition}
}

@article{Baez2002,
  author  = {Baez, John C.},
  title   = {The Octonions},
  journal = {Bulletin of the American Mathematical Society},
  volume  = {39},
  number  = {2},
  pages   = {145--205},
  year    = {2002},
  doi     = {10.1090/S0273-0979-01-00934-X}
}

@misc{LeanGA,
  author       = {Wieser, Eric and Song, Utensil},
  title        = {{lean-ga}: Geometric Algebra in {Lean}},
  howpublished = {\url{https://github.com/pygae/lean-ga}},
  year         = {2020},
  note         = {Lean library formalizing geometric algebra; Lean~3 era}
}

@misc{Mathlib,
  author       = {{The Mathlib Community}},
  title        = {{mathlib4}: The math library of {Lean} 4},
  howpublished = {\url{https://github.com/leanprover-community/mathlib4}},
  year         = {2024},
  note         = {Includes {\tt LinearAlgebra.CliffordAlgebra},
                  {\tt LinearAlgebra.ExteriorAlgebra},
                  {\tt spinGroup}, and {\tt pinGroup}}
}

@misc{SpinorLeanRepo,
  author       = {Ramos, Arthur F. and Hulak, David B.},
  title        = {{SpinorLean}: {Lean} Companion Repository for Exterior-Model
                  Spinors in Split Rank},
  howpublished = {\url{https://github.com/Arthur742Ramos/SpinorLean}},
  year         = {2026},
  note         = {GitHub repository}
}

@incollection{deMouraUllrich2021,
  author    = {de Moura, Leonardo and Ullrich, Sebastian},
  title     = {The {Lean} 4 Theorem Prover and Programming Language},
  booktitle = {Automated Deduction -- {CADE} 28},
  editor    = {Platzer, Andr{\'e} and Sutcliffe, Geoff},
  series    = {Lecture Notes in Computer Science},
  volume    = {12699},
  pages     = {625--635},
  publisher = {Springer},
  address   = {Cham},
  year      = {2021},
  doi       = {10.1007/978-3-030-79876-5_37}
}

@article{Wieser2022Clifford,
  author  = {Wieser, Eric and Song, Utensil},
  title   = {Formalizing Geometric Algebra in {Lean}},
  journal = {Advances in Applied Clifford Algebras},
  volume  = {32},
  number  = {3},
  pages   = {Article 28},
  year    = {2022},
  doi     = {10.1007/s00006-021-01164-1}
}

@inproceedings{MathlibCommunity2020,
  author    = {{The Mathlib Community}},
  title     = {The {Lean} Mathematical Library},
  booktitle = {Proceedings of the 9th {ACM} {SIGPLAN} International Conference
               on Certified Programs and Proofs},
  pages     = {367--381},
  publisher = {ACM},
  year      = {2020},
  doi       = {10.1145/3372885.3373824},
  note      = {Common scholarly reference for the mathlib ecosystem}
}

\end{document}